\documentclass[11pt,twoside]{article}
\usepackage{a4}
\usepackage{amssymb,amsmath,amsthm,latexsym}
\usepackage{amsfonts}
  \usepackage{amsfonts,hyperref}
\usepackage{graphicx}

\newtheorem{theorem}{Theorem}[section]

\newtheorem{corollary}[theorem] {Corollary}

\newtheorem{proposition}[theorem]{Proposition}
\newtheorem{remark}[theorem]{Remark}

\vspace{5cm}

\begin{document}
  
  \label{'ubf'}  
\setcounter{page}{1}                                 

\markboth {\hspace*{-9mm} \centerline{\footnotesize \sc
   Extensions of some results of Jovovic and Dhar  }
                 }
                { \centerline                           {\footnotesize \sc  
         Pankaj Jyoti Mahanta and Manjil P. Saikia                                                } \hspace*{-9mm}              
               }

\vspace*{-2cm}

\begin{center}
{ 
       {\Large \textbf { \sc  Extensions of some results of Jovovic and Dhar
                               }
       }
\\

\medskip

{\sc Mahanta, Pankaj Jyoti \& Saikia, Manjil P.}\\
{\footnotesize Gonit Sora, Dhalpur, Assam 784165, India}\\
{\footnotesize Mathematical and Physical Sciences division, School of Arts \& Sciences, Ahmedabad University, Navrangpura, Ahmedabad - 380009, Gujarat, India
}\\
{\footnotesize e-mail: {\it pankaj@gonitsora.com} \& {\it manjil@saikia.in}}
}
\end{center}

\thispagestyle{empty}

\hrulefill

\begin{abstract}  
{\footnotesize  We look at extensions of formulas given by Jovovic and recently proved by Dhar on integer partitions where the smallest part occurs at least $m$ times and on integer partitions with fixed differences between the largest and smallest parts where the smallest part occurs at least $k$ times. Our results extend Dhar's results for the $m=2$ and $k=1$ cases to the general cases for arbitrary $m$ and $k$. We also look at analogous results for overpartitions and $\ell$-regular partitions.
}
 \end{abstract}
 \hrulefill



\section{Introduction}\label{sec:intro}

A partition $\lambda$ of $n$ is a non-increasing sequence of natural numbers $\lambda_1\geq \lambda_2\geq \cdots \geq \lambda_k$, such that $\sum_{i=1}^k\lambda_i=n$, where $\lambda_i$'s are called the parts of the partition and $k$ is called the length of the partition. Here, $\lambda_1$ and $\lambda_k$ are called the largest and smallest parts of the partition respectively. It is customary to denote by $p(n)$ the number of partitions of $n$. Several interesting statistics on ordinary partitions and other generalizations have been studied since decades. Andrews \cite{andrews1998theory} wrote the standard reference book on partitions to which we refer the reader for more details.

Recently, Dhar \cite{dhar2021proofs} studied the statistic $a_m(n)$ (\href{https://oeis.org/A117989}{A117989}), which counts the number of partitions of $n$ where the smallest part occurs at least $m$ times, for $m=2$. He proved the following proposition using both analytic and combinatorial methods.
\begin{proposition}[Formula 1, \cite{dhar2021proofs}]\label{prop1}
For all natural numbers $n$, we have
\[a_2(n)=2p(n)-p(n+1).\]
\end{proposition}
\noindent He further proved the following result combinatorially.
\begin{proposition}[Formula 2, \cite{dhar2021proofs}]\label{prop2}
For all natural numbers $n$, we have
\[a_2(n)=p(2n,n),\]
where $p(m,n)$ is the number of partitions of $m$ with fixed difference between the largest and smallest parts equal to $n$.
\end{proposition}
\noindent The function $p(m,n)$ was studied by Andrews, Beck and Robbins \cite{andrews-pams} who gave a generating function for it.

In this paper we present two generalizations of Proposition \ref{prop1} to $a_m(n)$, from which Proposition \ref{prop1} follows as a corollary. The first generalization is proved combinatorially and is given below.
\begin{theorem}\label{G1}
For all $n\geq1$ and $m\geq2$, we have
\[a_m(n)=2p(n)-p(n+1)-p(n-2)+p(n-m)-\sum_{\ell=2}^{m-1}\sum_{k=3}^{\big\lfloor\frac{n}{\ell}\big\rfloor+1}\mathcal{Q}_{\ell,k}(n),\]
where $\mathcal{Q}_{\ell,k}(n)$ is the number of partitions of $n-\ell(k-1)$ with smallest part at least $k$.
\end{theorem}

\begin{remark}
Clearly Proposition \ref{prop1} is a corollary of this result.
\end{remark}
\begin{remark}
All empty sums are taken to be $0$ and all empty products are taken to be $1$ in this paper.
\end{remark}
The second generalization is proved analytically, and is stated below. This generalization answers the question that Dhar \cite{dhar2021proofs} asked in his paper about a closed form generating function of $a_m(n)$. We use the standard notations
\[
(a)_n=(a;q)_n:=\prod_{i=0}^{n-1}(1-aq^i),
\]
and
\[
(a)_\infty=(a;q)_\infty:=\lim_{n\rightarrow\infty}(a;q)_n
\] in the remainder of this paper.

\begin{theorem}\label{thm:am}
For $n,m>0$, we have
\[\sum_{n=1}^{\infty}a_m(n)q^n=\frac{1}{(q)_\infty}\sum_{j=0}^{m-1}(-1)^jq^{\binom{j+1}{2}-mj}(q^{m-j})_j-(-1)^{m-1}q^{-\binom{m}{2}}(q)_{m-1}. \]
\end{theorem}
\noindent It is easy to see that Proposition \ref{prop1} follows from Theorem \ref{thm:am}. Two further easy corollaries are given below.
\begin{corollary}\label{thm:a3}
We have
\[
a_3(n)=3p(n)-p(n+1)-2p(n+2)+p(n+3),
\]
for all natural numbers $n$.
\end{corollary}

\begin{corollary}\label{thm:a4}
We have
\[a_4(n)=4p(n)-p(n+1)-2p(n+2)-2p(n+3)+p(n+4)+2p(n+5)-p(n+6),\]
for all natural numbers $n$.
\end{corollary}
\noindent We prove Theorm \ref{thm:am} in Section \ref{sec:genprop1}.

With regards to Proposition \ref{prop2}, if we denote by $a_{m}(n,k)$ the total number of partitions of $n$ where the smallest part occurs at least $m$ times and the difference between the largest and smallest parts is $k$, then the following result is obvious (if the smallest part is $k$ and the largest part is $n+k$ for a partition counted by $a_{m-1}(2n,n)$, then just remove the part $n+k$ and add a part of size $k$ to get a partition counted by $a_m(n)$).

\begin{proposition}\label{prop3}
For all natural number $n$ and $m\geq 2$ we have,
\[a_m(n)=a_{m-1}(2n,n).\]
\end{proposition}

\begin{remark}
Clearly $a_{1}(2n,n)=p(2n,n)$ and hence Proposition \ref{prop2} follows as a corollary of Proposition \ref{prop3}.
\end{remark}
\noindent In Section \ref{sec:and} we will find the generating function of $a_m(n,\ell)$ for $\ell>1$. Specializing $m=1$ will give us the main result of Andrews, Beck and Robbins \cite{andrews-pams}.

After the work of Andrews, Beck and Robbins \cite{andrews-pams}, other authors such as Chern \cite{chern}, Chern and Yee \cite{chern-yee} and Lin and Xheng \cite{lin-zheng} looked at similar results for other classes of partitions, such as overpartitions and $\ell$-regular partitions. In a similar spirit, all of the statistics defined in this section can also be suitably modified and defined for other classes of partitions, such as overpartitions, $\ell$-regular partitions, etc. We discuss these briefly towards the end of the paper. The paper is organized as follows: in Section \ref{sec:proof1} we prove Theorem \ref{G1}, in Section \ref{sec:genprop1} we prove Theorem \ref{thm:am}, in Section \ref{sec:and} we find the generating function of $a_m(n,\ell)$ for $\ell>1$, in Section \ref{sec:over} we look at analogous results for overpartitions and $\ell$-regular partitions, and finally we end the paper with some concluding remarks and questions in Section \ref{sec:final}.

\section{Proof of Theorem \ref{G1}}\label{sec:proof1}

We define the following sets
\[
P(n):=\{\lambda |\lambda~\text{is a partition of}~n\},
\]
and
\[
A_m(n):=\{\lambda|\lambda\in P(n), ~\text{the smallest part occurs at least}~m~\text{times} \}.
\]

Let us focus on the set $A_3(n)$ initially. The set $P(n)-A_3(n)$ contains all partitions of $n$ in which the least part occurs exactly once or twice. Now we have three cases.

\textbf{Case 1. (The least part occurs exactly once)} Here we have $p(n+1)-p(n)$ such partitions. Let $d(n)=p(n+1)-p(n)$ and $c(n)$ denote the number of partitions of $n$ where the least part occurs exactly once. Then, $d(n)$ is the number of partitions of $n+1$ which do not contain $1$ as a part because every partition of $n+1$ which contains $1$ as a part can be obtained by adjoining $1$ as a part to every partition of $n$. Let $\mathcal{C}_n$ and $\mathcal{D}_n$ be the sets which are enumerated by $c(n)$ and $d(n)$ respectively. 

Let $\pi\in \mathcal{C}_n$: if $1$ is a part of $\pi$ then as it is the smallest possible part, it occurs exactly once. We add $1$ to the part $1$ of such a partition and obtain a partition in $\mathcal{D}_n$ with $2$ as the smallest part; otherwise if $1$ is not a part of $\pi$ then adding $1$ to the smallest part we again obtain a partition in $\mathcal{D}_n$.

Let $\mu \in \mathcal{D}_n$: since $\mu$ doesn't contain $1$ as a part so the smallest part is $\geq 2$. If the smallest part of $\mu$ occurs exactly once then we subtract $1$ from that part to get a partition in $\mathcal{C}_n$, otherwise we subtract $1$ from one of the smallest parts to get a partition in $\mathcal{C}_n$.

The above establishes a bijection between the sets $\mathcal{C}_n$ and $\mathcal{D}_n$, $n\geq 1$ and we are done.

\textbf{Case 2. (The least parts are $1$ and $1$)} If we remove the 1's, then they become partitions of $n-2$ in which there is no 1 as a part. So, here we have $p(n-2)-p(n-3)$ partitions.

\textbf{Case 3. (The least parts are greater than $1$)} Let the least parts of a partition be $k-1$ and $k-1$, for some $k\geq 3$. If we remove both the parts, then the partition of $n$ becomes a partition of $n-2(k-1)$ with the least part greater than or equal to $k$. If the least part of such partitions of $n$ is $k-1$, then here we have $\mathcal{Q}_k(n)$ partitions of $n-2(k-1)$, where $\mathcal{Q}_k(n)$ is the number of partitions of $n-2(k-1)$ with lowest part at least $k$. Therefore,
\[p(n)-a_3(n)=p(n+1)-p(n)+p(n-2)-p(n-3)+\sum_{k=3}^{\big\lfloor\frac{n}{2}\big\rfloor+1}\mathcal{Q}_k(n).\]
Hence,
\[a_3(n)=2p(n)-p(n+1)-p(n-2)+p(n-3)-\sum_{k=3}^{\big\lfloor\frac{n}{2}\big\rfloor+1}\mathcal{Q}_k(n).\]

Shifting out focus to the set $A_4(n)$ we see that to compute $a_4(n)$ two more cases arise.

\textbf{Case 4. (The least parts are $1, 1$, and $1$)} Similar to Case 2, here we have $p(n-3)-p(n-4)$ partitions.

\textbf{Case 5. (There are exactly three least parts and they are greater than $1$)} Here we have $\mathcal{Q}_{3,k}(n)$ partitions, corresponding to the partitions of $n$ with the least parts $k-1(>1)$, where $\mathcal{Q}_{\ell,k}(n)$ denote the number of partitions of $n-\ell(k-1)$ with smallest part $k$. Therefore,
\begin{multline*}
    p(n)-a_4(n)=p(n+1)-p(n)+p(n-2)-p(n-3)+p(n-3)-p(n-4)\\+\sum_{k=3}^{\big\lfloor\frac{n}{2}\big\rfloor+1}\mathcal{Q}_k(n)+\sum_{k=3}^{\big\lfloor\frac{n}{3}\big\rfloor+1}\mathcal{Q}_{3,k}(n).
\end{multline*}
Hence,
\[a_4(n)=2p(n)-p(n+1)-p(n-2)+p(n-4)-\sum_{k=3}^{\big\lfloor\frac{n}{2}\big\rfloor+1}\mathcal{Q}_{2,k}(n)-\sum_{k=3}^{\big\lfloor\frac{n}{3}\big\rfloor+1}\mathcal{Q}_{3,k}(n).\]

Thus, we observe that to compute $a_m(n)$ we need to add two more cases in addition to the cases used for computing $a_{m-1}(n)$. The new cases for the computation of $a_m(n)$ are as follows.

\textbf{Case $2m-4$. (The least parts are $m-1$ repeated $1$'s)} Similar to Cases 2 and 4 above, here we have $p(n-m+1)-p(n-m)$ partitions.

\textbf{Case $2m-3$. (There are exactly $m-1$ least parts and they are greater than $1$)} Here we have $\mathcal{Q}_{m-1,k}(n)$ partitions, corresponding to the partitions of $n$ with the least parts $k-1 (>1)$. Therefore,
\begin{multline*}
    p(n)-a_m(n)=p(n+1)-p(n)+\sum_{i=2}^{m-1}\big(p(n-i)-p(n-i-1)\big)\\+\sum_{\ell=2}^{m-2}\sum_{k=3}^{\big\lfloor\frac{n}{3}\big\rfloor+1}\mathcal{Q}_{\ell,k}(n)+\sum_{k=3}^{\big\lfloor\frac{n}{3}\big\rfloor+1}\mathcal{Q}_{m-1,k}(n).
\end{multline*}
This immediately gives us Theorem \ref{G1}.

\section{Generating Function of \texorpdfstring{$a_m(n)$}{}}\label{sec:genprop1}

In this section we generalize Proposition \ref{prop1} by finding an explicit form of the generating function of $a_m(n)$, thereby proving Theorem \ref{thm:am}. Before proceeding further, we state three results which we will need in our analysis.

\begin{proposition}[Heine?s transformation, \cite{andrews1998theory}, p.19, Corollary 2.3]\label{Heine}
If $|q|,|t|,|b|<1$, then
\[\sum_{k=0}^{\infty}\frac{(a)_k(b)_kt^k}{(q)_k(c)_k}=\frac{(b)_\infty(at)_\infty}{(c)_\infty(t)_\infty}\sum_{k=0}^{\infty}\frac{(c/b)_k(t)_kb^k}{(q)_k(at)_k}.\]
\end{proposition}

\begin{proposition}[$q$-binomial theorem, \cite{andrews1998theory}, p. 36]\label{q-binom}
    We have
    \[
(z)_n=\sum_{j=0}^n\binom{n}{j}_q(-1)^jz^jq^{j(j-1)/2},
\]
where
\[
\binom{a}{b}_q:=\frac{(q)_a}{(q)_b(q)_{a-b}}.
\]
\end{proposition}

\begin{proposition}[Cauchy?s identity, \cite{andrews1998theory}, p.17, Theorem 2.1]
If $|q|<1,|t|<1$, then
\begin{equation}\label{Cauchy}
    \sum_{k=0}^{\infty}\frac{(a)_kt^k}{(q)_k}=\frac{(at)_\infty}{(t)_\infty}.
\end{equation}
\end{proposition}
\noindent By replacing $a=0$ in \eqref{Cauchy}, we get
\begin{equation}\label{eq:cor-cauchy}
    \sum\limits_{k=0}^{\infty}\dfrac{t^k}{(q)_k}=\dfrac{1}{(t)_\infty}.
\end{equation}

The generating function of $a_m(n)$ is given by
\begin{align*}
\sum_{n=1}^{\infty} a_m(n)q^n &= \sum_{k=1}^{\infty} q^{\underbrace{k+k+\cdots+k}_{m}}(1+q^k+q^{2k}+\cdots)(1+q^{k+1}+q^{2(k+1)}+\cdots)\cdots\\ \nonumber
&=\sum_{k=1}^\infty\frac{q^{mk}}{(q^k)_\infty}\\&=\frac{1}{(q)_\infty}\sum_{k=1}^\infty(q)_{k-1}q^{mk}\\&=\frac{q^m}{(q)_\infty}\sum_{k=0}^{\infty}(q)_{k}q^{mk}\\&=\frac{q^m}{(q)_\infty}\sum_{k=0}^{\infty}\frac{(q)_{k}(q)_{k}q^{mk}}{(q)_{k}}\\ \nonumber
&=\frac{q^m(q^{m+1})_\infty}{(q^m)_\infty}\sum_{k=0}^{\infty}\frac{(q^m)_{k}q^k}{(q)_{k}(q^{m+1})_{k}}\\ \nonumber
& \hspace{20pt} (\text{by replacing} \ a=q,b=q,c=0, \ \text{and} \ t=q^m \ \text{in Proposition \ref{Heine}})\\ \nonumber
&=\sum_{k=0}^{\infty}\frac{q^{k+m}}{(q)_{k}(1-q^{k+m})}.
\end{align*}

We are now ready to prove Theorem \ref{thm:am}.
\begin{proof}[Proof of Theorem \ref{thm:am}]
From the above discussion, we have
\begin{align*}
\sum_{n=1}^{\infty} a_m(n)q^n &= \sum_{k=0}^{\infty}\frac{q^{k+m}}{(q)_{k}(1-q^{k+m})}\\&=\sum_{k=0}^\infty \frac{q^{k+m}(q^{k+1})_{m-1}}{(q)_{k+m}}\\
&=\sum_{k=0}^\infty \frac{q^{k}(q^{k-m+1})_{m-1}}{(q)_{k}}-\sum_{k=0}^{m-1} \frac{q^{k}(q^{k-m+1})_{m-1}}{(q)_{k}}.
\end{align*}
\noindent Applying Proposition \ref{q-binom} to $(q^{k-m+1})_{m-1}$ in the first term above and simplifying the second term above, we obtain the following.
\begin{align*}
\sum_{n=1}^{\infty} a_m(n)q^n &= \sum_{k=0}^\infty\frac{q^k}{(q)_k}\sum_{j=0}^{m-1}\binom{m-1}{j}_q(-1)^jq^{(k-m+1)j}q^{j(j-1)/2}-(q^{-m+1})_{m-1}\\
&=\sum_{k=0}^\infty\frac{q^k}{(q)_k}\sum_{j=0}^{m-1}(-1)^jq^{\binom{j+1}{2}-mj}\frac{q^{kj}(q)_{m-1}}{(q)_j(q)_{m-1-j}}-(-1)^{m-1}q^{-\binom{m}{2}}(q)_{m-1}.
\end{align*}
\noindent Using equation \eqref{eq:cor-cauchy} and simplifying further, we obtain the following.
\begin{align*}
\sum_{n=1}^{\infty} a_m(n)q^n &= \frac{1}{(q)_\infty}\sum_{j=0}^{m-1}(-1)^jq^{\binom{j+1}{2}-mj}(q^{m-j})_j-(-1)^{m-1}q^{-\binom{m}{2}}(q)_{m-1}.
\end{align*}
\end{proof}

Corollaries \ref{thm:a3} and \ref{thm:a4} now follow from Theorem \ref{thm:am} by replacing $m=3$ and $m=4$ respectively. We leave the details to the reader.

\section{Generating Functions of \texorpdfstring{$a_m(n,\ell)$}{}}\label{sec:and}

We follow the method used by Andrews, Beck and Robbins \cite{andrews-pams} to find the generating function of $a_m(n, \ell)$ below.

\begin{theorem}\label{thm:and}
    For $\ell>1$ we have
    \begin{multline*}
        \sum_{n=1}^{\infty} a_m(n, \ell)q^n =\frac{q^{\ell+m+1}(q)_{m}(q)_{\ell-m+1}}{((q)_\ell)^2}(-1)^{-m-1}q^{(-m^2-3m-2)/2}\\ \times \left((q)_{\ell}- \sum_{j=0}^{m}\binom{\ell}{j}_q (-1)^{j}q^{j+j(j-1)/2}\right).
    \end{multline*}
\end{theorem}

\begin{proof}
Clearly the generating function of $a_m(n,\ell)$ is given by
\begin{align*}
\sum_{n=1}^{\infty} a_m(n, \ell)q^n &= \sum_{k=1}^{\infty} q^{\underbrace{k+k+\cdots+k}_{m}} q^{k+\ell} \prod_{i=0}^\ell (1+q^{k+i}+q^{2(k+i)}+\cdots)\\ 
&=q^\ell \sum_{k\geq 1}\frac{q^{(m+1)k}(q)_{k-1}}{(q)_{k+\ell}}= q^{\ell+m+1} \sum_{k\geq 0}\frac{q^{(m+1)k}(q)_k}{(q)_{k+\ell+1}}\\ 
&=\frac{q^{\ell+m+1}}{(q)_{\ell+1}}\sum_{k\geq 0}\frac{(q)_k(q)_kq^{(m+1)k}}{(q)_k(q^{\ell+2})_k}.
\end{align*}
We now use the following transformation \cite[p. 38]{andrews1998theory}
\[
\sum_{k\geq 0}\frac{(a)_k(b)_kz^k}{(q)_k(c)_k}=\frac{(c/b)_\infty (bz)_\infty}{(c)_\infty (z)_\infty}\sum_{j\geq 0}\frac{(abz/c)_j(b)_j(c/b)^j}{(q)_j(bz)_j},
\]
to get
\begin{align*}
\sum_{n=1}^{\infty} a_m(n, \ell)q^n &= \frac{q^{\ell+m+1}(q^{\ell+1})_\infty (q^{m+2})_\infty}{(q)_{\ell+1}(q^{\ell+2})_\infty (q^{m+1})_\infty} \sum_{j\geq 0}\frac{(q^{m+1-\ell})_j(q)_jq^{(\ell+1)j}}{(q)_j(q^{m+2})_j}\\
&=\frac{q^{\ell+m+1}}{(q)_\ell}\sum_{j=0}^{\ell-m-1}\frac{(q^{m+1-\ell})_jq^{(\ell+1)j}}{(q^{m+1})_{j+1}}\\
&=\frac{q^{\ell+m+1}}{(q)_\ell}\\& \times \sum_{j=0}^{\ell-m-1}\frac{(1-q^{\ell-m-1})(1-q^{\ell-m-2})\cdots (1-q^{\ell-m-j})(-1)^jq^{(m+1)j+\binom{j+1}{2}}}{(q^{m+1})_{j+1}}\\
&=\frac{q^{\ell+m+1}(1-q)\cdots (1-q^m)}{(1-q^{\ell-m})\cdots (1-q^\ell)}\sum_{j=0}^{\ell-m-1}\frac{(-1)^jq^{(m+1)j+\binom{j+1}{2}}}{(q)_{m+j+1}(q)_{\ell-m-j-1}}\\
&=\frac{q^{\ell+m+1}(q)_{m}(q)_{\ell-m+1}}{(q)_\ell(q)_\ell}\sum_{j=0}^{\ell-m-1}\binom{\ell}{m+j+1}_q (-1)^jq^{(m+1)j+\binom{j+1}{2}} ,
\end{align*}
\noindent We use Proposition \ref{q-binom} to get
\begin{align*}
\sum_{n=1}^{\infty} a_m(n, \ell)q^n &=\frac{q^{\ell+m+1}(q)_{m}(q)_{\ell-m+1}}{((q)_\ell)^2}\sum_{j=m+1}^{\ell}\binom{\ell}{j}_q (-1)^{j-m-1}q^{(m+1)(j-m-1)+\binom{j-m}{2}} \\
&=\frac{q^{\ell+m+1}(q)_{m}(q)_{\ell-m+1}}{((q)_\ell)^2}(-1)^{-m-1}q^{(-m^2-3m-2)/2}\\& \quad \times \sum_{j=m+1}^{\ell}\binom{\ell}{j}_q (-1)^{j}q^{j+j(j-1)/2}\\
&=\frac{q^{\ell+m+1}(q)_{m}(q)_{\ell-m+1}}{((q)_\ell)^2}(-1)^{-m-1}q^{(-m^2-3m-2)/2}\\
&\qquad \times \left((q)_{\ell}- \sum_{j=0}^{m}\binom{\ell}{j}_q (-1)^{j}q^{j+j(j-1)/2}\right).
\end{align*}
\end{proof}

If we replace $m=1$ in Theorem \ref{thm:and} we get back the result of Andrews, Beck and Robbins \cite[Theorem 1]{andrews-pams}.

\section{Other Classes of Partitions}\label{sec:over}

In this section we briefly look at overpartitions and $\ell$-regular partitions and prove results analogous to those stated in Section \ref{sec:intro}.

\subsection{Overpatitions}

Overpartitions of $n$ are the partitions of $n$ in which the first occurrence (equivalently, the last occurrence) of a part may be overlined. The number of overpartitions of $n$ are denoted by $\bar p(n)$. For example,  $\bar{p}(3)=8$, and the overpartitions of 3 are
\[3,\bar{3},2+1,\bar{2}+1,2+\bar{1},\bar{2}+\bar{1},1+1+1,\bar{1}+1+1.\]
We have the generating function
\[\sum_{n=0}^{\infty}\bar{p}(n)q^n = \prod_{n=1}^{\infty}\frac{1+q^n}{1-q^n}=1 + 2q + 4q^2 + 8q^3 + 14q^4 +\cdots.\]

We define $\bar{a}_m(n)$ to be the number of overpartitions of $n$ where the smallest part occurs at least $m$ times. We get the generating function of $\bar{a}_m(n)$ as follows
\[\sum_{n=1}^{\infty} \bar{a}_m(n)q^n=\sum_{k=1}^{\infty} \dfrac{2q^{mk}(-q^k)_\infty}{(1+q^k)(q^k)_\infty}.\]
Here, an overlined part is equal to a non-overlined part if their value is equal. For example $\bar{1}=1$. The terms $q^{mk}$, $(-q^k)_\infty$ and $(q^k)_\infty$ are similarly explained as in finding the generating function for $a_m(n)$. We have the factor of $2$ because the smallest part can be either overlined or non-overlined and finally we need to divide by $(1+q^k)$ because if $k$ is the smallest part then it would not be generated by $(-q^k)_\infty$.

Analogous to Proposition \ref{prop1}, we now have the following result.

\begin{theorem}\label{over:a2}
For all $n\geq 1$, we have
\[\bar{a}_2(n)=2\bar{p}(n)-\bar{p}(n+1)+\bar{u}(n+1),\]
where $\bar{u}(n)$ is the number of overpartitions of $n$ with the following conditions:
\begin{enumerate}
    \item $1$ is not a part.
    \item Smallest part must be overlined, and no other part of that value is present. (For example, $7+7+\bar{2}$ is included, but $7+7+\bar{2}+2$ is not.)
    \item The value of the first and second greatest parts are equal or consecutive, but if they are consecutive then the second greatest part must be overlined.  (For example, $7+\bar{6}+\bar{2}$ is included, but $7+6+\bar{2}$ is not.)
\end{enumerate}
\end{theorem}

\begin{proof}
Clearly $\bar{p}(n)-\bar{a}_2(n)$ counts the overpartitions of $n$ where the least part occurs exactly once. And $\bar{p}(n+1)-\bar{p}(n)$ counts the overpartitions of $n+1$ where there is no 1 as a part. Let us denote the sets by $\bar{A}$ and $\bar{B}$, respectively, whose cardinalities are counted by $\bar p(n)-\bar a_2(n)$ and $\bar p(n+1)-\bar p(n)$ respectively.

We divide $\bar{A}$ into two classes:
\begin{enumerate}
    \item[(a)] \textbf{The least part is not overlined.} In this case we add 1 to the least part. Then for each of these overpatitions there corresponds a unique overpartition in $\bar{B}$. The least part of these overpartitions in $\bar{B}$ are not overlined. Conversely, for each overpartition in $\bar{B}$ with the least part not overlined, if we subtract 1 form the least part, then there corresponds a unique overpartition in $\bar{A}$. Let us denote this class of $\bar{B}$ by $\bar{B}_1$.
    \item[(b)] \textbf{The least part is overlined.} In this case we add 1 to the largest part. Then for each of these overpatitions there corresponds a unique overpartition in $\bar{B}$. Let us denote this class of $\bar{B}$ by $\bar{B}_2$. Conversely, we subtract 1 from the largest part of the overpatitions in $\bar{B}_2$.
\end{enumerate}
So, in $\bar{B}-\bar{B}_1\bigcup \bar{B}_2$, we are left with the overpartitions, where,
\begin{enumerate}
    \item the least part is overlined,
    \item the first and second largest parts are equal, or they are consecutive with the property that the second largest part is overlined. (Note that, the later case is necessary because of the converse part of the above Case (b). Since, if we subtract 1 from the largest part, then the overpartitions become of the form $a+\bar{a}+\cdots$ or $\bar{a}+\bar{a}+\cdots$.)
\end{enumerate}
Therefore, $|\bar{B}-\bar{B}_1\bigcup \bar{B}_2|=\bar{u}(n+1)$.
Hence, $\bar{p}(n)-\bar{a}_2(n)=\bar{p}(n+1)-\bar{p}(n)-\bar{u}(n+1)$.
\end{proof}

For the sake of completeness, the generating function for $\bar u(n)$ is
\begin{align*}
\sum_{n=1}^{\infty} \bar{u}(n)q^n =& \sum_{k=1}^{\infty} \sum_{t=1}^{\infty} \frac{q^k(1+q^{k+1})(1+q^{k+2})\cdots(1+q^{k+t-1})(q^{2(k+t)}+q^{2(k+t)})}{(1-q^{k+1})(1-q^{k+2})\cdots(1-q^{k+t-1})(1-q^{k+t})}\\
& + \sum_{k=1}^{\infty} \sum_{t=2}^{\infty} \frac{q^k(1+q^{k+1})(1+q^{k+2})\cdots(1+q^{k+t-2})q^{k+t-1}(q^{k+t}+q^{k+t})}{(1-q^{k+1})(1-q^{k+2})\cdots(1-q^{k+t-1})}\\
& + \sum_{k=1}^{\infty}q^k(q^{k+1}+q^{k+1})\\
=& \sum_{k=1}^{\infty} \sum_{t=1}^{\infty} \frac{2q^{3k+2t}(-q^{k+1};q)_{t-1}}{(q^{k+1};q)_{t}} + \sum_{k=1}^{\infty} \sum_{t=2}^{\infty} \frac{2q^{3k+2t-1}(-q^{k+1};q)_{t-2}}{(q^{k+1};q)_{t-1}} + \sum_{k=1}^{\infty}2q^{2k+1}\\
=& 2\sum_{k=1}^{\infty}\bigg( \frac{q^{2k+1}}{1-q^{k+1}} + \sum_{t=2}^{\infty} q^{3k+2t-1}(1+q)\frac{(-q^{k+1};q)_{t-2}}{(q^{k+1};q)_{t}} \bigg).
\end{align*}

Let us now define $\bar p(n, t)$ to be the number of overpartitions of $n$ where the difference between the largest and smallest parts equal $t$. Analogous to Proposition \ref{prop2} we have the following result.
\begin{theorem}\label{thm:over-1}
For all $n\geq 1$, we have
\[2\bar{a}_2(n)=\bar{p}(2n,n).\]
\end{theorem}

\begin{proof} Following a similar method as Dhar \cite{dhar2021proofs}, here we get a one-to-two correspondence between the overpartitions counted by $\bar{a}_2(n)$ and the overpatitions counted by $\bar{p}(2n,n)$.

We add $n$ to the rightmost smallest part of each overpartition in the set counted by $\bar a_2(n)$. For example, if $\bar{k}+k$ is a part in a partition in this set, then we add $n$ to $k$. Then the new overpartition belongs to the set counted by $\bar p(2n,n)$ and its largest part is not overlined, which is greater than all other parts. So corresponding to this overpartition there is another unique overpatition in the second set, all of whose parts are same except now the largest part is overlined.
\end{proof}

If we define $\bar a_m(n, \ell)$ to be the number of overpartitions of $n$ where the smallest part occurs at least $m$ times and the largest part minus the smallest part is $\ell$, then using the above method we get the following theorem, analogous to Proposition \ref{prop3}.
\begin{theorem}
For all $n\geq 1$ and $m\geq 2$, we have
\[2\bar{a}_m(n)=\bar{a}_{m-1}(2n,n).\]
\end{theorem}
\noindent Theorem \ref{thm:over-1} follows from this when $m=2$.

\subsection{\texorpdfstring{$\ell$}{}-regular partitions}

The partitions of $n$ with no parts divisible by $\ell$ are called $\ell$-regular partitions, and the total number of such partitions is denoted by $b_{\ell}(n)$. We have the following generating function
\[\sum_{n=0}^{\infty}b_{\ell}(n)q^n = \prod_{n=1}^{\infty}\frac{1-q^{\ell n}}{1-q^n}= \frac{(q^\ell;q^\ell)_\infty}{(q)_{\infty}}= \frac{1}{(q,q^2,\ldots,q^{\ell-1};q^\ell)_\infty},\]
where
\[
(a_1, a_2, \ldots, a_k; q)_\infty:=\prod_{i=1}^k(a_i; q)_\infty.
\]

Let us define $a_{m(\ell)}(n)$ to be the number of $\ell$-regular partitions of $n$ where the smallest part occurs at least $m$ times. We get the generating function of $a_{m(\ell)}(n)$ as follows
\begin{align*}
\sum_{n=1}^{\infty} a_{m(\ell)}(n)q^n =& \sum_{k=0}^{\infty} \sum_{t=1}^{\ell-1} \frac{q^{(\ell k+t)m}}{\prod_{r=1}^{t-1}(q^{\ell(k+1)+r};q^\ell)_\infty \times \prod_{r=t}^{\ell-1}(q^{\ell k+r};q^\ell)_\infty}\\
=& \sum_{k=0}^{\infty} \sum_{t=1}^{\ell-1} \frac{q^{(\ell k+t)m} (q^{\ell k+1};q)_{t-1}} {(q^{\ell k+1},q^{\ell k+2},\ldots,q^{\ell k+\ell-1};q^\ell)_\infty}.
\end{align*}
If $\ell=2$, then the partitions contain no even parts. In this case, we get
\[\sum_{n=1}^{\infty} a_{m(2)}(n)q^n=\frac{q^m}{(q;q^2)_\infty}\sum_{k=0}^{\infty}(q;q^2)_kq^{2km}.\]

Analogous to Proposition \ref{prop1}, we now have the following result.

\begin{theorem}\label{regular:a2}
For all $n\geq 1$, we have
\[a_{2(2)}(n)=b_2(n)+b_2(n+1)-b_2(n+2).\]
\end{theorem}
\begin{proof}
The quantity  $b_2(n)-a_{2(2)}(n)$ counts the $2$-regular partitions of $n$ in which the least part occurs exactly once. Again, $b_2(n+2)-b_2(n+1)$ counts the $2$-regular partitions of $n+2$ in which there is no $1$ as a part. These two quantities are equal. We add $2$ to the least part of the $2$-regular partitions counted by $b_2(n)-a_{2(2)}(n)$. And conversely we subtract 2 from the least part of the 2-regular partitions given by $b_2(n+2)-b_2(n+1)$.
\end{proof}

Let us now define $a_{m(\ell)}(n, k)$ to be the number of $\ell$-regular partitions of $n$ where the smallest part occurs at least $m$ times and the largest part minus the smallest part is $k$, then like before we get the following theorem, analogous to Propositions \ref{prop2} and \ref{prop3}.

\begin{theorem} If $m,\ell\geq2$, and $n$ is divisible by $\ell$, then
\[a_{2(\ell)}(n)=b_\ell(2n,n).\]
More generally,
\[a_{m(\ell)}(n)=a_{m-1(\ell)}(2n,n).\]
\end{theorem}
\noindent This result is not true if $n$ is not divisible by $\ell$, since the sum of the smallest part and $n$ may be divisible by $\ell$. For example, $b_2(2n,n)=0$ if $n$ is odd. But, we get the following more general result in this case.
 
\begin{theorem} If $n$ is odd, then
 \[a_{2(2)}(n)=b_2(2n+1,n+1).\]
 
 More generally, let $m,\ell\geq2$, and $n$ is not divisible by $\ell$. If $n=\ell k+r$, for some interger $k$ and $1\leq r<\ell$, then
\[a_{m(\ell)}(n)=a_{(m-1)(\ell)}(2n+\ell-r,n+\ell-r).\]
\end{theorem}

\section{Concluding Remarks}\label{sec:final}

Several natural questions arise from the work we have described so far. Some of these we list below.
\begin{enumerate}
    \item Proposition \ref{prop3} is a very simple generalization. Is it possible to find $a_m(n)$ in terms of partition functions such as $p(an,bn)$, where $a,b$ are positive integers?
    \item Is it possible to find generating functions of analogues of $a_m(n)$ (and other statistics defined here) for other partition functions such as $(\ell,m)$-regular partitions, $t$-core partitions, partition with designated summands, $k$-colored partitions, etc.?
    \item Dhar \cite{dhar2021proofs} pointed out that the generating function of $p(2n,n)$ is still not found in a `nice' closed form. Similarly, the generating function of $a_m(2n,n)$ is also not found here. Can we find these?
    \item Is it possible to find $q$-series proofs of the results in Section \ref{sec:over} which are proved combinatorially?
    \item Is it possible to find $a_{m(2)}(n)$, $a_{2(\ell)}(n)$, and $a_{m(\ell)}(n)$ for $m,\ell>2$?
    \item If an overlined part is not equal to a non-overlined part, even if their value is equal (for example $\bar{1}\neq 1$), then the following is the generating function of $\bar{a}_m(n)$
\[\sum_{n=1}^{\infty} \bar{a}_m(n)q^n=\sum_{k=1}^{\infty} \dfrac{q^{mk}(-q^k)_\infty}{(q^k)_\infty}.\]
Then, can we find $\bar{a}_2(n)$ and $\bar{a}_m(n)$?
\item We can prove for $n\geq 1$, \[p(2n,n)=1+p(n-2)+\sum_{m=2}^{\lfloor\frac{n}{3}\rfloor} p^\star_m(n-2m),\]
where $p^\star_m(n-2m)$ is the number of partitions of $n-2m$ with the least part greater than or equal to $m$. Are there any interesting properties of $a_m(n,\ell)$ that can be proved using this relation?
\item Andrews, Beck and Robbins \cite[Theorem 2]{andrews-pams} also give a generalization of Theorem \ref{thm:and} ($m=1$) to partitions with a set of specified distances. It would be interesting to explore this direction with some of the statistics defined in this paper.
\item Breuer and Kronholm \cite{BreuerKronholm} extended the result of Andrews, Beck and Robbins \cite{andrews-pams} to partitions where the fixed difference between the largest and smallest parts is at most a fixed integer. Chapman \cite{chapman} gave a combinatorial proof of this result. It would be interesting to extend this setting for the statistics defined in this paper.
\end{enumerate}

\section*{Acknowledgements}

The authors thank the anonymous referee for several helpful comments, specially regarding the presentation of the proof of Theorem \ref{G1} and the proof of Theorem \ref{thm:am}, where we used the referee's arguments.

\label{'ubl'}  
\end{document}